\documentclass[11pt]{elsarticle}
\usepackage{amsmath,amssymb,amsthm}
\usepackage{mathrsfs}
\usepackage{graphicx}
\usepackage{tikz}
\usepackage{enumerate}
\usepackage{float}
\usepackage[normalem]{ulem}

\usetikzlibrary{calc}
\theoremstyle{plain}
\newtheorem{theorem}{Theorem}[section]
\newtheorem{lemma}[theorem]{Lemma}
\newtheorem{proposition}[theorem]{Proposition}
\theoremstyle{definition}
\newtheorem{definition}[theorem]{Definition}
\newtheorem{remark}[theorem]{Remark}
\newtheorem{corollary}[theorem]{Corollary}
\newtheorem{example}[theorem]{Example}
\newtheorem{problem}[theorem]{Problem}

\makeatletter

\newcommand{\End}[1]{[#1\to #1]}
\newcommand{\id}{\mathrm{id}}
\newcommand{\cl}{\mathrm{cl}}
\newcommand{\fin}{\mathrm{fin}}
\newcommand{\intt}{\mathrm{int}}
\newcommand{\rmnum}[1]{\romannumeral #1}
\newcommand{\Rmnum}[1]{\expandafter@slowromancap\romannumeral #1@}
\makeatother
\newcommand{\dda}{\mathord{\mbox{\makebox[0pt][l]{\raisebox{-.4ex}{$\downarrow$}}$\downarrow$}}}
\newcommand{\dua}{\mathord{\mbox{\makebox[0pt][l]{\raisebox{.4ex}{$\uparrow$}}$\uparrow$}}}

\newcommand{\da}{{\downarrow}}
\newcommand{\ua}{{\uparrow}}
\makeatletter
\def\ps@pprintTitle{%
  \let\@oddhead\@empty
  \let\@evenhead\@empty
  \def\@oddfoot{\reset@font\hfil\thepage\hfil}
  \let\@evenfoot\@oddfoot
}\makeatother

\begin{document}
\begin{frontmatter}

\title{Forbidden substructures for coherence of domains \tnoteref{t1}}
\tnotetext[t1]{Supported by NSF of China (No.12371457).}

\author[1]{Rongqi Xiao}
\ead{rongqixiao2024@163.com}

\author[1]{Xiaodong Jia\corref{a1}}
\ead{jiaxiaodong@hnu.edu.cn}
\cortext[a1]{Corresponding author.}
\address[1]{School of Mathematics, Hunan University, Changsha, Hunan, 410082, China}

\begin{abstract}
A coherent domain, in the sense of Dana Scott's domain theory, is a domain in which the intersection of every two compact saturated subsets is again compact, when the domain is equipped with the Scott topology. Coherence plays key roles in classifying Cartesian closed subcategories of domains and in characterizing Lawson compactness of domains.   

In this paper, we find two typical domains that fail to be coherent, and prove that a bounded algebraic domain fails to be coherent if and only if it has one of the two typical domains as its Scott-continuous retract. Similar results also generalize to bounded continuous domains, provided that the domains in consideration are hereditarily Lindel\"of and weakly Hausdorff in the Scott topology. 

\end{abstract}

\end{frontmatter}
\section{Introduction}
The characterization of certain mathematical properties by forbidden substructures is an old but charming theme. In graph theory, the celebrated Wagner's theorem tells us that a graph is planar if and only if one \emph{cannot} find copies of $K_5$ (the complete graph on 5 vertices) or of $K_{3,3}$ (the complete bipartite graph on 3+3 vertices) as a minor in it. In lattice theory, the $M_3$-$N_5$ theorem shows that a lattice $L$ is  distributive if and only if $N_5$ or $M_3$ \emph{cannot} lattice-embed in $L$. 
Results about forbidden substructures in Dana Scott's domain theory are of particular interest to non-Hausdorff topologists and theoretical computer scientists, as they are extensively used in analyzing the order-theoretic and topological structures of higher function spaces of domains, which are in turn applied to semantics of programming languages. This typical line of research can be easily spotted in A.Jung's celebrated Ph.D thesis \cite{jung88c}, where he manoeuvred forbidden substructures of domains {\color{black}skillfully} and eventually proved that there are exactly two maximal Cartesian closed subcategories of pointed domains, laying a fundamental theory for the theory of semantic categories. Forbidden structures for meet-continuous dcpo's (directed-complete posets), in particular, are considered in the Ph.D thesis of the second author \cite{jia2018meet}, and by using his developed forbidden structure results he and coauthors were able to prove that \textcolor{black}{the category of domains and that of  quasi-continuous domains} share the same Cartesian closed subcategories, solving a question posed by Goubault-Larrecq \cite{goubault2010omega, goubault12}. Other works related to forbidden substructures in domain theory can be found in \cite{XI2018126, zhang96, liu1996solutions}, to name a few. 

In domain theory, coherence is an indispensable topological property in characterizing Lawson compactness of dcpo's: a dcpo is Lawson compact if and only if it is well-filtered, finitely grounded and coherent \cite{xi2017well, jia2018meet}. \textcolor{black}{A topological space is \emph{coherent} if binary intersections of compact saturated subsets are again compact}, which is automatically valid in Hausdorff topologies but not in general non-Hausdorff ones. A dcpo is called coherent if it is coherent as a topological space in the Scott topology. As domains are automatically well-filtered, coherence coincides with Lawson compactness on domains with finitely many minimal elements. That reveals the importance of coherence. 

In this paper, we try to characterize coherence of domains {\color{black}in terms of order-theoretic forbidden substructures.} Explicitly, we give two classes of forbidden substructures for coherence on domain structures, and show that a bounded algebraic domain{\color{black}, i.e., an algebraic domain with a least
element $\perp$ and a greatest element $\top$} fails to be coherent if and only if it contains, as a Scott-continuous retract, one of the copies of the forbidden substructures. Moreover, we generalize that result to bounded domains, this time with \textcolor{black}{hereditarily Lindel\"ofness} and  weak Hausdorffness assumed. {\color{black}We conclude the paper with some open questions.}

\section{Preliminaries}
The following definitions can be found in \cite{abramsky94, gierz03, goubault13a}.

Let $L$ be a poset and $X$ a subset of $L$. Let $X^{u}=\{b\in L\mid$ for every $x\in X,~b\geqslant x\}$, $\da X=\{b\in L \mid ~$for some $x\in X,~ b\leqslant x\}$ and $\ua X=\{b\in L\mid~ $for some $x\in X,~ b\geqslant x\}$. {\color{black}A subset $A$ of $L$ is called an \emph{anti-chain} if any two distinct elements of $A$ are incomparable. For subsets $A\subseteq B\subseteq L$, $A$ is \emph{cofinal in $B$} if for every $b\in B$ there exists $a\in A$ with $b\leqslant a$.} A subset $D$ of $L$ is called \emph{directed} if it is non-empty and for all $d_1,d_2\in D$, there exists $d_3\in D$ such that $\ d_3\geqslant d_1,d_2$. Dually, a subset $F$ of $L$ is called \emph{filtered} if it is non-empty and for all $f_1,f_2\in F$, there exists $f_3\in F$ such that $f_3\leqslant f_1,f_2$. A poset $L$ is called a \emph{directed-complete poset $($dcpo$)$} if every directed subset of $L$ has a supremum in $L$. A dcpo $L$ is called a \emph{bicomplete dcpo} if its order dual $L^{op}$ is a dcpo as well. For any two elements $x,y$ in $L$, $x$ is \emph{way below} $y$, denoted by $x\ll y$, if for every directed subset $D\subseteq L$ with existing supremum $\sup D$, that $y\leqslant \sup D $ always implies that  $x\in {\da D}$. Let $\dda x=\{y\in L \mid y\ll x\}$ and $\dua x=\{y\in L \mid x\ll y\}$. A poset is called \emph{continuous} if for every $x\in L$, the set $\dda x$ is directed and $x=\sup \dda x$. Continuous dcpo's are also called \emph{continuous domains} or simply \emph{domains}. Let $L$ be a dcpo. An element $x\in L$ is called \emph{compact} if $x \ll x$, and the subset of all compact elements in $L$ is denoted by $K(L)$. A dcpo $L$ is called \emph{algebraic} if for every $x\in L, \da x\cap K(L)$ is directed and $\sup (\da x\cap K(L))=x$, and algebraic dcpo's are also called \emph{algebraic domains}. Algebraic domains are instances of domains. 

A subset $U$ of $L$ is \emph{Scott open} if $U=\ua U$ and for every  directed subset $D$ for which $\sup D$ exists, $\sup D\in U$ implies $D\cap U\neq \varnothing$. The collection of all Scott open subsets of $L$ is a topology called the \emph{Scott topology} of $L$ and will be denoted by $\sigma (L)$. The space $(L,\sigma (L))$ is usually referred as a \emph{Scott space} and written as $\Sigma L$. {\color{black}A topological space is \emph{second countable} if its topology has a countable base. A domain $L$ is said to be \emph{second countable} if the Scott space $\Sigma L$ is second countable. Without further reference, we always equip $L$ with the
Scott topology $\sigma(L)$.} When $L$ is a domain, $\dua x$ is always Scott open for every $x\in L$. 
Let $L$ and $M$ be two dcpo's, a function $f:L \rightarrow M$ is called \emph{Scott-continuous} if it is monotone and for every directed subset $D$ of $L$, $f(\sup D)=\sup f(D)$. A map $f:L \rightarrow M$ is Scott-continuous if and only if $f:\Sigma L \rightarrow \Sigma M$ is continuous in the topological sense. {\color{black} A dcpo $X$ is called \emph{meet-continuous} if for every $x\in X$ and every directed subset $D\subseteq X$ with $x\leq \bigvee D$, one has $x\in \cl_{\sigma}\big(\da x\cap\da D\big)$. Domains are always meet-continuous \cite{gierz03}.}

Let $X$ be a $T_0$ topological space. We denote $\mathcal{O}(X)$ the set of all open subsets of $X$, ordered by the inclusion order $\subseteq$. The \emph{specialization order} $\leqslant_\tau$ on $X$ is defined by $x\leqslant_\tau y$ if and only if $x\in \cl(\{y\})$, where $\cl(\{y\})$ is the closure of $\{y\}$. 
Subsets of $X$ that are upper sets in the specialization order are called \emph{saturated}. A subset $K$ of $X$ is \emph{compact} if for all open cover $\left \{ U_i\right \}_{i\in I}$ of $K$, one can always extract a finite subcover $\left\{ U_j\right \}_{j\in J}$ of $K$ where $J\subseteq^{\fin} I$. Compactness can also be characterized by using directed open families: a subset $K$ of $X$ is \emph{compact} if and only if for all directed families of opens $\left \{ U_i\right \}_{i\in I}$, that $K\subseteq \textstyle\bigcup_{i\in I} U_i$ implies $\ K\subseteq U_{i_0}$ for some $i_0\in I$. A topological space $X$ is \emph{locally compact} if for every $x\in X$ and $U$ an open set with $x\in U$, there is a compact saturated subset $Q\subseteq X$ such that $x\in \intt(Q)$ and $Q \subseteq U$, {\color{black}where $\intt(Q)$ is the topological interior of $Q$.}
 
\textcolor{black}{We denote $\mathcal{Q}(X)$ the set} of all compact saturated subsets of $X$ with the Smyth order $\sqsubseteq $ defined by $A\sqsubseteq B$ if and only if $B \subseteq A$. {\color{black}A $T_0$ space $X$ is said to be \textit{well-filtered} if for every open set $U$ and every filtered family $\mathcal F\subseteq Q(X)$ (under the inclusion order), $\bigcap \mathcal F\subseteq U$ implies that $F\subseteq U$ for some $F\in\mathcal F$.} In a well-filtered space, the intersection of a filtered family of compact saturated subsets is again compact saturated. 
 A topological space $X$ is said to be \emph{coherent} if for any two compact saturated subsets $Q_1,Q_2$ of $X$, $Q_1\cap Q_2$ is compact (and it's automatically saturated). 
A non-empty subset $A$ of a topological space $X$ is said to be \emph{irreducible} if for any finite family $\left \{ C_f\right \}_{f\in F}$ of closed sets, that $A\subseteq \textstyle\bigcup_{f\in F} C_f$ implies that $A\subseteq C_{f_0}$ for some $f_0$. A topological space is \emph{sober} if every irreducible closed subset $A$ is the closure of a unique element $a\in X,\ i.e.,\ A=\cl(\{a\})$ for a unique $a\in X$. Sober spaces are well-filtered \cite{gierz03}.

\begin{proposition}\label{b2}\cite{gierz03}
Every continuous domain $L$ is a sober, hence well-filtered, and \textcolor{black}{locally compact space in the Scott topology. In particular}, for every Scott open set $U$ of $\Sigma L$ and $x\in U$, there exists $y\in U$, such that $x\in \dua y\subseteq \ua y \subseteq U$.
\end{proposition}

\begin{definition}\label{a3}
A \emph{retract} of a topological space $Y$ is a  topological space $X$ such that there are two continuous maps $s: X\rightarrow Y$(section) and $r:Y\rightarrow X$(retraction) such that $r\circ s=\id_X $. 
\end{definition}
For dcpo's $X,Y$, we say that $X$ is a \emph{Scott-retract} of $Y$ if and only if $\Sigma X$ is a retract of $\Sigma Y$. 
Many of the aforementioned topological properties on spaces such as {\color{black}continuity}, sobriety, well-filteredness and coherence, are preserved by retractions, together with bi-completeness, an order-theoretic property. 

\subsection{Bicompleteness}

We recall that a dcpo $L$ is called a \emph{bicomplete dcpo} if its order dual $L^{op}$ is a dcpo.

\begin{proposition}\label{a4}\cite{jia2018meet}
Let $L$ be a bicomplete dcpo and dcpo $M$ {\color{black}be} a retract of $L$. Then $M$ is bicomplete.
\end{proposition}

In the following example, the three dcpo's fail to be bicomplete for obvious reasons. 

\begin{example}\label{a5}
The following examples are algebraic domains but none of them is bicomplete:
\begin{itemize}
    \item Let $\mathbb{N}$ be the natural numbers with the usual order and consider its dual poset $\mathbb{N}^{op}$, then $\mathbb{N}^{op}$ is not bicomplete;
    \item Let $\mathcal{K}(\mathbb{N}^{op})=\mathbb{N}^{op} \cup \left \{ a,b\right \}$, where $a$ and $b$ are two incomparable elements that are not in $\mathbb{N}^{op}$ and they are both strictly smaller than any element in $\mathbb{N}^{op}$, then $\mathcal{K}(\mathbb{N}^{op})$ is not bicomplete;
    \item Let $\mathcal{K}(\mathbb{N}^{op})_{\bot}$ be the dcpo obtained by adding a least element $\bot$ to $\mathcal{K}(\mathbb{N}^{op})$, then $\mathcal{K}(\mathbb{N}^{op})_{\bot}$ is not bicomplete.
\end{itemize}

\begin{figure}[H]
\begin{center}
\begin{tikzpicture}

\coordinate[shape=circle, fill=white,draw=black, inner sep=1pt] (a) at (-2, 0); 
\coordinate[shape=circle, fill=white,draw=black, inner sep=1pt] (b) at (2, 0);  
\coordinate[shape=circle, fill=white,draw=black, inner sep=1pt] (c) at (0, -1.5);  

 \coordinate[] (Zn1) at (0, 1);  
 \coordinate[] (Zn2) at (0, 1.5);  
 \coordinate[] (Zn3) at (0, 2);  
 \coordinate[shape=circle, fill=white,draw=black, inner sep=1pt] (zn-1) at (0, 2.5);  
 \coordinate[shape=circle, fill=white,draw=black, inner sep=1pt] (zn) at (0, 3);  
 \coordinate[shape=circle, fill=white,draw=black, inner sep=1pt] (zn+1) at (0, 3.5);  
 \coordinate[shape=circle, fill=white,draw=black, inner sep=1pt] (zn+2) at (0, 4);  

    \draw (a) -- (c);
    \draw (b) -- (c);
    \draw (zn+2) -- (zn+1);
    \draw (zn) -- (zn+1);
    \draw (zn-1) -- (zn);
    \draw[dashed] (zn-1) -- (Zn1);

    \node[left] at (a) {$ a $};
    \node[right] at (b) {$ b $};
    \node[left] at (c) {$ \bot $};
    \node[left] at (zn+2) {$ 0 $};
    \node[left] at (zn+1) {$ 1 $};
    \node[left] at (zn) {$ 2 $};
    \node[left] at (zn-1) {$ 3 $};

\end{tikzpicture}
\\Figure 1: $\mathcal{K}(\mathbb{N}^{op})_{\bot}$ in Example \ref{a5}
\end{center}
\end{figure}
\end{example}

The above three non-bicomplete examples could be generalized in a similar fashion. The second author \cite{jia2019dichotomy} used their generalized forms to give a characterization theorem for bicompleteness, on which our results rely.

\begin{definition} \label{a1}
A chain $C$ is said to be \emph{well-ordered} if every non-empty subset $A\subseteq C$ has a least element and we say a chain $C$ is \emph{downward well-ordered} if its order dual chain $C^{op}$ is well-ordered.
\end{definition}

\begin{definition}\label{a6}
For every downward well-ordered chain $C$ without a bottom element, we define 
\begin{itemize}
    \item $\mathcal{K}(C)=C\cup \left \{ a,b\right \}$, where $a$ and $b$ are two incomparable elements below $C$, and
    \item $\mathcal{K}(C)_{\bot}$ to be the lifting of $\mathcal{K}(C)$ by adding a bottom element $\bot$.
\end{itemize}
\end{definition}

\begin{theorem}\label{a7}\cite{jia2019dichotomy}
Let $L$ be a sober dcpo. If every minimal element in $L$ is a compact element, then the following are equivalent:
\begin{enumerate}[(i)]
\item $L$ is not bicomplete;
\item $L$ has some $C$, $\mathcal{K}(C)$ or $\mathcal{K}(C)_\bot$ as a Scott-retract, where $C$ is a downward well-ordered chain without a bottom.
\end{enumerate}
\end{theorem}

\begin{remark}\label{bicompelte_bot}
    When the sober dcpo $L$ in the above theorem has a least element, then $L$ is not bicomplete if and only if $L$ has $\mathcal{K}(C)_\bot$ as a retract. 
\end{remark}

\subsection{Coherence}

We recall that a topological space $X$ is said to be \emph{coherent} if  for any two compact saturated subsets $Q_1,Q_2$ of $X$, $Q_1\cap Q_2$ is compact. Again, a dcpo $L$ is said to be coherent if and only if it is coherent in the Scott topology. The following lemma given in \cite{jia16a} simplifies the verification of coherence on well-filtered dcpo's.

\begin{lemma}\cite{jia16a}\label{b4}
Let $L$ be a well-filtered dcpo. Then $L$ is coherent if and only if $\ua x\cap \ua y$ is compact for all $x,y\in L$.
\end{lemma}

As we have mentioned above, coherent spaces are preserved by retractions. The same result applies to coherent dcpo's. {\color{black}The proof is basically borrowed from~\cite[Proposition 2.17]{jung04}, where it is treated for general spaces. 

\begin{proposition}\cite{jung04}\label{b6}
Let $L$ be a coherent dcpo. If a dcpo $M$ is a retract of $L$, then $M$ is coherent.
\end{proposition}
\begin{proof}
Let $s:M\to L$ and $r:L\to M$ be Scott-continuous maps with $r\circ s=\mathrm{id}_M$. For $Q_1,Q_2\in \mathcal Q(M)$, the sets
$\ua_L s(Q_1)$ and $\ua_L s(Q_2)$ are compact saturated in $L$.
Since $L$ is coherent, $K=\ua_L s(Q_1)\cap\ua_L s(Q_2)$ is compact. Moreover,
$r(K)=Q_1\cap Q_2.$ To prove $Q_1\cap Q_2\subseteq r(K)$, we let $q\in Q_1\cap Q_2$. Since $s$ is injective, we have 
$s(q)\in \ua_L s(Q_1)\cap\ua_L s(Q_2)=K$, and hence $q=r(s(q))\in r(K)$. Conversely, let $m\in r(K)$. Then $m=r(x)$ for some $x\in K$. Since
$x\in\ua_L s(Q_1)$, there exists $q_1\in Q_1$ such that
$s(q_1)\leq x$. By monotonicity of $r$, $q_1=r(s(q_1))\leq r(x)=m$.
As $Q_1$ is saturated, it follows that $m\in Q_1$. Similarly, from $x\in\ua_L s(Q_2)$, we obtain $m\in Q_2$. Thus $m\in Q_1\cap Q_2$, and so $r(K)\subseteq Q_1\cap Q_2$. Hence $Q_1\cap Q_2$ is compact, being the continuous image of $K$ along $r$. So $M$ is coherent.
\end{proof}}

\section{Coherence on algebraic domains}

In this section, we will try to give concrete forbidden structures to characterize (non)coherent algebraic domains. First, we use Lemma~\ref{b4} to further simplify the verification of coherence on algebraic domains. 

\begin{lemma}\label{b5}
Let $L$ be an algebraic domain, then the following are equivalent:
\begin{enumerate}[(i)]
\item $\Sigma L$ is coherent;
\item $\ua x\cap \ua y$ is compact for all $x, y\in L$;
\item $\ua x\cap \ua y$ is compact for all $x, y\in K(L)$, where $K(L)$ is the set of all compact elements of $L$.
\end{enumerate}
\end{lemma}
\begin{proof}
($\rmnum1\Leftrightarrow \rmnum2$) Algebraic domains are sober and hence well-filtered in the Scott topology. This equivalence is a special case of Lemma~\ref{b4}.

($\rmnum2\Rightarrow \rmnum3$) Obvious.

($\rmnum3\Rightarrow \rmnum2$) 
For $x, y \in L$, we realize that $\ua x\cap \ua y = \bigcap_{b\in \da y \cap K(L)}\bigcap_{a\in \da x \cap K(L)} (\ua a \cap \ua b)$, and the collection $\{\ua a\cap \ua b\mid a\in \da x \cap K(L), b\in \da y \cap K(L)\}$ is a filtered family of compact saturated sets, by assumption. Hence $\ua x\cap \ua y$ is compact since $L$ is well-filtered. 
\end{proof}

In the light of the above lemma, we could easily find domain structures that fail to be coherent.

\begin{example}\label{b7}
The following examples are algebraic domains but none of them is coherent:
\begin{itemize}
    \item Let $W(\overline{\mathbb{N}})=\overline{\mathbb{N}}\cup \left \{ a,b\right \}$ where $\overline{\mathbb{N}}$ is the anti-chain of natural numbers with the discrete order and $a$ and $b$ are two incomparable elements below $\overline{\mathbb{N}}$, then $W(\overline{\mathbb{N}})$ is not coherent. To see this, take the compact saturated subsets $\ua a$ and $\ua b$. Their intersection is $\overline{\mathbb{N}}$ and it is not a compact subset since for each $n\in \overline{\mathbb{N}}$, $n$ is a compact element and hence $\ua n = \{n\}, n\in \overline{\mathbb{N}}$ form a Scott open cover of $\overline{\mathbb{N}}$, and any finitely many of them fail to cover $\overline{\mathbb{N}}$. 
    \item Let $W(\overline{\mathbb{N}})_\bot$be the dcpo obtained by adding a least element $\bot$ to $W(\overline{\mathbb{N}})$, then $W(\overline{\mathbb{N}})_\bot$ is not coherent.
    \item Let $W(\overline{\mathbb{N}})_\bot^\top$ be the dcpo obtained by adding a top element \rotatebox[origin=c]{180}{$\bot$} to $W(\overline{\mathbb{N}})_{\bot}$, then $W(\overline{\mathbb{N}})_\bot^\top$ is not coherent.
\end{itemize}

\begin{figure}[H]
\begin{center}
\begin{tikzpicture}

    \coordinate[shape=circle, fill=white,draw=black, inner sep=1pt](z1) at (-3, 2);     
    \coordinate[shape=circle, fill=white,draw=black, inner sep=1pt] (z2) at (-2, 2);   
    \coordinate[shape=circle, fill=white,draw=black, inner sep=1pt] (z3) at (-1, 2);    
    \coordinate[shape=circle, fill=white,draw=black, inner sep=1pt] (a) at (-2, 0); 
    \coordinate[shape=circle, fill=white,draw=black, inner sep=1pt] (b) at (2, 0);  
    \coordinate[shape=circle, fill=white,draw=black, inner sep=1pt] (c) at (0, -1);  
    \coordinate[shape=circle, fill=white,draw=black, inner sep=1pt] (d) at (0, 4);  

 \coordinate[shape=circle, fill=black,draw=black, inner sep=0.5pt] (Zn1) at (-0.25, 2);  
 \coordinate[shape=circle, fill=black,draw=black, inner sep=0.5pt] (Zn2) at (0, 2);  
 \coordinate[shape=circle, fill=black,draw=black, inner sep=0.5pt] (Zn3) at (0.25, 2);  
 \coordinate[shape=circle, fill=white,draw=black, inner sep=1pt] (zn) at (1, 2);  
 \coordinate[shape=circle, fill=white,draw=black, inner sep=1pt] (zn+1) at (2, 2);  
 \coordinate[shape=circle, fill=white,draw=black, inner sep=1pt] (zn+2) at (3, 2);  
 \coordinate[shape=circle, fill=black,draw=black, inner sep=0.5pt] (Zn4) at (3.25, 2);  
 \coordinate[shape=circle, fill=black,draw=black, inner sep=0.5pt] (Zn5) at (3.5, 2);  
  \coordinate[shape=circle, fill=black,draw=black, inner sep=0.5pt] (Zn6) at (3.75, 2);  

    \draw (a) -- (z1);
    \draw (a) -- (z2);
    \draw (a) -- (z3);
    \draw (a) -- (zn);
    \draw (a) -- (zn+1);
    \draw (a) -- (zn+2);
    \draw (b) -- (z1);
    \draw (b) -- (z2);
    \draw (b) -- (z3);
    \draw (b) -- (zn);
    \draw (b) -- (zn+1);
    \draw (b) -- (zn+2);
    \draw (c) -- (a);
    \draw (c) -- (b);
    \draw (d) -- (z1);
    \draw (d) -- (z2);
    \draw (d) -- (z3);
    \draw (d) -- (zn);
    \draw (d) -- (zn+1);
    \draw (d) -- (zn+2);
    \node[below] at (a) {$ a $};
    \node[below] at (b) {$ b $};
    \node[below] at (c) {$ \bot $};
    \node[above] at (d) {$ \top $};
    \node[above] at (z1) {$ 0 $};
    \node[above] at (z2) {$ 1 $};
    \node[above] at (z3) {$ 2 $};
    \node[above] at (zn) {$ n $};
    \node[above] at (zn+1) {$ n+1 $};
    \node[above] at (zn+2) {$ n+2 $};

\end{tikzpicture}
\\Figure 2: $W(\overline{\mathbb{N}})_\bot^\top$ in Example~\ref{b7}
\end{center}
\end{figure}

\end{example}

\begin{definition}\label{webt}
For every anti-chain $E$ with infinitely many elements, we define
\begin{itemize}
    \item $W(E)=E\cup \left \{ a,b\right \}$ where $a,b$ are two incomparable elements below $E$,
    \item $W(E)_\bot$ to be the lifting of $W(E)$ by adding a bottom element $\bot$, and 
    \item $W(E)_\bot^\top$ by adding a top element \rotatebox[origin=c]{180}{$\bot$} to $W(E)_\bot$.
\end{itemize}  
\end{definition}

Now we arrive at one of our main results.

\begin{theorem}\label{b8}
Let $L$ be an algebraic domain with bottom and top elements $\bot$ and \rotatebox[origin=c]{180}{$\bot$}, then the following are equivalent:
\begin{enumerate}[(i)]
\item $L$ is not coherent;
\item $L$ has $\mathcal{K}(C)_\bot$ (Definition~\ref{a6}) or $W(E)_{\bot}^{\top}$(Definition~\ref{webt}) as a retract, where $C$ is a downward well-ordered chain without a bottom and $E$ is an anti-chain with infinitely many elements.
\end{enumerate}
\end{theorem}
\begin{proof}
($\rmnum2\Rightarrow \rmnum1$) This can be obtained directly from Proposition \ref{b6}, since neither $\mathcal{K}(C)_\bot$ nor $W(E)_\bot^\top$ is coherent. 

($\rmnum1\Rightarrow \rmnum2$) Assume that $L$ is not coherent, then we use Lemma \ref{b5} to find $a, b\in K(L)$ such that $\ua a\cap \ua b$ is not compact in $L$. Now we consider the following two cases. 

$\mathbf{Case\ 1}$: There exists a maximal chain $C'$ in $\ua a\cap \ua b$, and $C'$ does not have a least element.  \textcolor{black}{This implies that $L$ is actually not bicomplete, and Remark~\ref{bicompelte_bot} enables us to find some $\mathcal K(C)_\bot$ as a retract in $L$. In order to keep the paper self-contained,  we speak the proof in full. }

\textcolor{black}{By applying the fact that every chain has a cofinal well-ordered subset
\cite[Section~17]{halmos1960naive}
to the dual chain $(C')^{op}$, we may choose a subset $C\subseteq C'$ such that
$C$ is downward well-ordered and, for each $c'\in C'$, there exists $c\in C$ with
$c\leq c'$.} We let $M = C \cup \{a,b,\bot\}$ and give $M$ the induced order from $L$. Clearly, $M$ is a copy of $\mathcal K(C)_\bot$. Figure~3 displays this copy\footnote{\color{black}Figure~3 illustrates the decomposition used in Case~1: the two labelled regions represent $\ua a$ and $\ua b$, and their common upper part contains the chosen cofinal chain $C$ onto which the retraction $r$ projects.}
. We define a map $r\colon L\to M$ as follows:

\begin{center}
$r(x)= \begin{cases}
\inf \left \{ c\in C\mid x\leqslant c  \right \}, & x\in \ua a \cap \ua b\\
a, & x\in \ua a \setminus \ua b\\
b, & x\in \ua b \setminus \ua a\\
\bot, & \mbox{otherwise.}
\end{cases}$
\end{center}

\begin{figure}[H]
\begin{center}
\begin{tikzpicture}
    \coordinate[shape=circle, fill=white,draw=black, inner sep=1pt] (A) at (-2, -2.5); 
    \coordinate[shape=circle, fill=white,draw=black, inner sep=1pt] (B) at (2, -2.5);  
    \coordinate (C) at (0, -1.5);  
    \coordinate[shape=circle, fill=white,draw=black, inner sep=1pt] (D) at (0, -3.833333333333333333);  

    \draw (A) -- (C) -- (B) -- (D) -- cycle;
    \draw (A) -- (D);
    \draw (B) -- (C);
    \draw (A) -- (-5, -0.5);
    \draw (B) -- (5, -0.5);
    \draw (C) -- (-3, 0);
    \draw (C) -- (3, 0);
    \node[below left] at (A) {$ a $};
    \node[below right] at (B) {$ b $};
    \node at (-3.05,-1.10) {$\ua a$};
    \node at ( 3.05,-1.10) {$\ua b$};
    \node[below] at (D) {$\bot$};

    \coordinate[shape=circle, fill=white,draw=black, inner sep=1pt](T) at (0, 2);     
    \coordinate[shape=circle, fill=white,draw=black, inner sep=1pt] (X1) at (0, 1.5);   
    \coordinate[shape=circle, fill=white,draw=black, inner sep=1pt] (X2) at (0, 1);    
    \coordinate[shape=circle, fill=white,draw=black, inner sep=1pt] (X3) at (0, 0.5);    
    \coordinate[] (...) at (0, 0.25); 
    \coordinate[] (...) at (0, 0); 
    \coordinate[] (...) at (0, -0.25); 
    
    \draw (T) -- (X1);
    \draw (X1) -- (X2);
    \draw (X2) -- (X3);
    \draw[dashed] (X3) -- (0, -0.25);

    \node[left] at (T) {$ $};
    \node[left] at (X1) {$ $};
    \node[left] at (X2) {$  $};
    \node[left] at (X3) {$ $};

\end{tikzpicture}
\\Figure 3: The situation in the proof of $\mathbf{Case\ 1}$
\end{center}
\end{figure}

We first show that $r(x)$ is well-defined. \textcolor{black}{The value of  $r(x)$ is the largest element of $C$ if $\left \{ c\in C\mid x\leqslant c  \right \}$ is empty}. Moreover, we notice that for each $x\in \ua a\cap \ua b$, $x$ cannot be below all elements in $C$, i.e., the set $\{c\in C \mid x \not\leq c\}$ is not empty; otherwise, {\color{black}$x\leq c'$ for every $c'\in C'$. Hence 
$C'\cup\{x\}$ would be a chain in $\ua a\cap\ua b$. By the maximality of 
$C'$, we would have $x\in C'$, and then $x$ would be the least element of $C'$, 
a contradiction to the fact that $C'$ does not have a least element.} Let $c_0 \in C$ be such that $x\not \leq c_0$. So $c_0$ is a lower bound of the set $\{c\in C \mid x \leq c\}$. {\color{black}Since $C$ is downward well-ordered, the set of all 
such lower bounds has a greatest element. This greatest lower bound is precisely $\inf\{c\in C \mid x\leq c\}$. }

Second, it is quite straightforward to see that the map $r$ is monotone. We prove that $r$ is even Scott-continuous. To this end, we take a directed subset $D$ of $L$. As $r$ is monotone, we know that $\sup r(D) \leq r(\sup D)$. We continue to prove $r(\sup D)\leq \sup r(D)$. That is just a case study: 
\begin{itemize}
    \item If $r(\sup D) = \bot $, {\color{black}then $r(\sup D)=\bot\leq \sup r(D)$.}
    \item If $r(\sup D) = a$, then by definition $\sup D \in \ua a \setminus \ua b$, and since $a$ is a compact element, we would know some $d'$ in $D$ is already in $\ua a \setminus \ua b$, and $r(d') = a$. So $r(\sup D) = a =r(d') \leq \sup r(D)$. The same arguments also apply to the case $r(\sup D) = b$. 
    \item If $r(\sup D) = \inf\{c\in C \mid \sup D \leq c\}$ is in $C$, then we know that $\sup D$ is in $\ua a\cap \ua b$, and \textcolor{black}{since $a,b$ are compact, after replacing $D$ by the cofinal directed subset $D\cap\ua a\cap\ua b$,} without loss of generality we may assume that $D\subseteq \ua a\cap \ua b$. 
    As $C$ is a downward well-ordered chain, $\sup r(D) = \max r(D)$ is in $C$, and we know $\max r(D) = r(d^*)$ for some $d^*\in D$. By monotonicity of $r$, $r(d) = r(d^*)$ for all $d\in D$ with $d\geq d^*$. \textcolor{black}{Put $m=r(d^*)$. For $d\in \ua d^*\cap D$, write $S_d=\{c\in C \mid d\leq c\}$. Then $\inf S_d=m$; hence, if $S_d$ has a least element, it must be $m$. Thus the following two cases are exhaustive.} The first case is that \textcolor{black}{$S_d$ has the least element $m$} for all $d\in \ua d^*\cap D$. This implies that $\sup D = \sup(\ua d^*\cap D)\leq \textcolor{black}{m}$, and by definition $r(\sup D) \leq \textcolor{black}{m} = \sup r(D)$. The second case is that \textcolor{black}{for some $d'\in \ua d^*\cap D$, the set $S_{d'}$ has no least element}. \textcolor{black}{Since $S_{d'}$ is an upper subset of $C$ with infimum $m$, we have $S_{d'}=\{c\in C \mid m<c\}$. For all $d\in D$ with $d\geq d'$, we have $S_d\subseteq S_{d'}$ and $\inf S_d=m$, hence $S_d=\{c\in C \mid d\leq c\}=\{c\in C \mid m<c\}$.} Therefore, $\{c\in C \mid \textcolor{black}{m}<c\} \subseteq \{c\in C \mid \sup D \leq c\}$, which implies that $r(\sup D) \leq \textcolor{black}{m} = \sup r(D)$, as required.

\end{itemize} 
So we have proved that $r$ is Scott-continuous. It is a retraction from $L$ to $M$, with the obvious section $s\colon M\to L$ defined as the inclusion of $M$ in to $L$. This inclusion itself is Scott-continuous since each element in $M$ is actually a compact element: $a, b$ and $\bot$ are obviously compact elements, and every element $c\in C \subseteq M$ is compact due to {\color{black}the fact} that $C$ is downward well-ordered.  

$\mathbf{Case\ 2}$: Every maximal chain $C$ in $\ua a\cap \ua b$ has a least  element.  

We observe that the least element of each maximal chain $C$ in $\ua a\cap \ua b$ 
is minimal in $\ua a\cap \ua b$ and is a compact element. The minimality of such elements is obvious. {\color{black}To see that they are compact, let $x$ be a minimal element of
$\ua a\cap\ua b$. Remember that $a,b$ are compact,
$\ua a \cap \ua b$ is Scott open. As $L$ is algebraic, $\da x\cap K(L)$ is directed and its supremum $x$ lies in the Scott open set $\ua a\cap\ua b$, then there exists
$y\in\da x\cap K(L)\cap(\ua a\cap\ua b)$. By the
minimality of $x$ in $\ua a\cap\ua b$, we have $x=y$ is compact.} We denote all the minimal elements of $\ua a\cap \ua b$ as $Z$. Since each element is contained in some maximal chain, every $x\in \ua a \cap \ua b$ is actually above some minimal and compact element $z\in Z$. Moreover, as $\ua a \cap \ua b$ is not a compact subset, $Z$ is actually an infinite subset of $L$.

\begin{figure}[H]
\begin{center}
\begin{tikzpicture}
    \coordinate[shape=circle, fill=white,draw=black, inner sep=1pt] (A) at (-2, -2.5); 
    \coordinate[shape=circle, fill=white,draw=black, inner sep=1pt] (B) at (2, -2.5);  
    \coordinate (C) at (0, -1.5);  
    \coordinate[shape=circle, fill=white,draw=black, inner sep=1pt] (D) at (0, -3.83333333333333);  

    \draw (A) -- (C) -- (B) -- (D) -- cycle;
    \draw (A) -- (D);
    \draw (B) -- (C);
    \draw (A) -- (-5, -0.5);
    \draw (B) -- (5, -0.5);
    \draw (C) -- (-3, 0);
    \draw (C) -- (3, 0);
    \node[below left] at (A) {$ a $};
    \node[below right] at (B) {$ b $};
    \node at (-3.05,-1.10) {$\ua a$};
    \node at ( 3.05,-1.10) {$\ua b$};
    \node[below] at (D) {$\bot$};

    \coordinate[shape=circle, fill=white,draw=black, inner sep=1pt](T) at (0, 2);     
    \coordinate[shape=circle, fill=white,draw=black, inner sep=1pt] (X1) at (-1, 1);   
    \coordinate[shape=circle, fill=white,draw=black, inner sep=1pt] (X2) at (1, 1);    
    \coordinate[shape=circle, fill=white,draw=black, inner sep=1pt] (Z1) at (-2, 0); 
    \coordinate[shape=circle, fill=white,draw=black, inner sep=1pt] (Z2) at (-1, 0); 
    \coordinate[shape=circle, fill=white,draw=black, inner sep=1pt] (Z3) at (-0.5, 0); 
    \coordinate[shape=circle, fill=white,draw=black, inner sep=1pt] (Zn) at (1, 0);  

 \coordinate[shape=circle, fill=black,draw=black, inner sep=0.5pt] (Z4) at (0.15, 0);  
 \coordinate[shape=circle, fill=black,draw=black, inner sep=0.5pt] (Z5) at (0.25, 0);  
  \coordinate[shape=circle, fill=black,draw=black, inner sep=0.5pt] (Z6) at (0.35, 0);  
 \coordinate[shape=circle, fill=black,draw=black, inner sep=0.5pt] (Zn1) at (1.35, 0);  
 \coordinate[shape=circle, fill=black,draw=black, inner sep=0.5pt] (Zn2) at (1.45, 0);  
 \coordinate[shape=circle, fill=black,draw=black, inner sep=0.5pt] (Zn3) at (1.55, 0);  

  \coordinate[shape=circle, fill=black,draw=black, inner sep=0.5pt] (x4) at (0.15, 1);  
 \coordinate[shape=circle, fill=black,draw=black, inner sep=0.5pt] (x5) at (0.25, 1);  
  \coordinate[shape=circle, fill=black,draw=black, inner sep=0.5pt] (x6) at (0.35, 1);  
 \coordinate[shape=circle, fill=black,draw=black, inner sep=0.5pt] (xn1) at (1.35, 1);  
 \coordinate[shape=circle, fill=black,draw=black, inner sep=0.5pt] (xn2) at (1.45, 1);  
 \coordinate[shape=circle, fill=black,draw=black, inner sep=0.5pt] (xn3) at (1.55, 1);  

    \draw[dashed] (T) -- (X1);
    \draw[dashed] (T)  -- (Zn);
    \draw[dashed] (T) -- (Z2);
    \draw[dashed] (T) -- (Z3);

    \draw (X2) -- (Z3);
    \draw (X1) -- (Z1);
    \draw (X2) -- (Zn);

    \node[above] at (T) {$\rotatebox[origin=c]{180}{$\bot$}$};
    \node[above] at (X1) {$ x_1 $};
    \node[above] at (X2) {$ x_2 $};
    \node[below] at (Z1) {$ z_1 $};
    \node[below] at (Z2) {$ z_2 $};
    \node[below] at (Z3) {$ z_3 $};
    \node[below right] at (Zn) {$ z_n $};

\end{tikzpicture}
\\Figure 4: The situation in the proof of $\mathbf{Case\ 2}$
\end{center}
\end{figure}

Now, we classify points in $\ua a\cap \ua b$ into two categories $F_1$ and $F_2$, where each element $x\in F_1$ is greater than or equal to {\color{black}exactly} one element $z_x \in Z$, and $F_2$ consists of the {\color{black}remaining} elements of $\ua a\cap \ua b$. 
We define $N = Z\cup \{\bot, a, b, \top\}$ and induce the order of $L$ onto it. 
Clearly $N$ is a copy of $\textcolor{black}{W(Z)^\top_\bot}$. Figure~4 displays this copy\footnote{\color{black}Figure~4 illustrates the situation in Case~2, where the elements $z_i\in Z$ are the minimal elements of $\ua a\cap\ua b$, and hence each $z_i$ lies above both $a$ and $b$.
The dashed lines from the $z_i$'s to $\top$ indicate only the order relations $z_i<\top$, not necessarily covering relations; there may be other elements between $z_i$ and $\top$, whose possible order relations are not specified in the picture.}. Now consider the following map $r\colon L\to N$:

\begin{center}
$r(x)= \begin{cases}
\top, & x\in F_2 \\
z_x, & x\in F_1 \\
a, & x\in \ua a\setminus \ua b\\
b, & x\in \ua b\setminus \ua a\\
\bot, & \mbox{otherwise.}
\end{cases}$
\end{center}

We prove that $r$ is Scott-continuous. It is monotone {\color{black}by definition}. Now we take a directed subset $D\subseteq L$ and prove that $r(\sup D) \leq \sup r(D)$. We consider the following cases:
\begin{itemize}
    \item If $r(\sup D) = \bot$, {\color{black}then $r(\sup D)=\bot\leq \sup r(D)$.}
    \item If $r(\sup D) = a$, then $\sup D$ is in $\ua a\setminus \ua b$, and since $a$ is compact some $d'\in D$ has to be in $\ua a\setminus \ua b$, so we know that $r(d') = a$ and $r(\sup D)= a =r(d') \leq \sup r(D)$. The same reasoning applies to the case $r(\sup D) = b$.
    \item If $r(\sup D) = z$ for $z\in Z$, then we know that $\sup D$ is in $\ua a\cap \ua b$ and $z$ is the unique element in $Z$ that is below $\sup D$. Hence some element $d\in D$ is already {\color{black}larger than} $z$, as $z$ is compact. For that $d$, we have $z\leq d \leq \sup D$. So $d\in F_1$ and $z=z_d$. Hence $r(d) = z$, and $r(\sup D) \leq \sup r(D)$.
    \item If $r(\sup D) = \top$, then that means $\sup D$ is above at least two distinct elements $z_1, z_2$ in~$Z$. Again, since $z_1$ and $z_2$ are compact and $D$ is directed, there exists a $d' \in D$ such that $z_1\leq d'$ and $z_2\leq d'$. Hence, $r(\sup D) = \top = r(d') \leq \sup r(D)$.
\end{itemize}
So we have proved that $r$ is Scott-continuous. It follows that $r:L\rightarrow N$ is a Scott-continuous retraction with the canonical inclusion of $N$ into $L$ as the corresponding section. {\color{black}Since $N$ has finite height, every directed subset of $N$ has a largest element; hence each element in $N$ is actually compact and the inclusion is Scott-continuous.}
\end{proof}

\begin{remark}\label{b9}
\
\begin{enumerate}
    \item {\color{black}For any dcpo $L$, $L^\top$ is coherent if and only if $L$ is coherent if and only if $L_\bot$ is coherent. Indeed, the compact saturated subsets of $L_\bot$ are precisely $L_\bot$ and those of $L$, while the nonempty compact saturated subsets of $L^\top$ are precisely $Q\cup\{\top\}$ with $Q\in \mathcal Q(L)$. Also, the element $\top$ is a compact element in $L^\top$, and we have that $L$ is algebraic if and only if $L^\top$ is algebraic if and only if $L_\bot$ is algebraic. }
    This means that for any given algebraic domain $L$, we only need to consider whether $L^\top_\bot$ has some $\mathcal{K}(C)_\bot$ or $W(E)_\bot^\top$ as a retract or not. That will suffice to  reflect the coherence of $L$. 
    \item {\color{black}Theorem~\ref{b8} holds for bounded countable domains since countable domains are algebraic; see~\cite{jia2026separating}.}
\end{enumerate}
\end{remark}

Recall that the Lawson topology on a dcpo is generated, as closed sets, by the Scott closed subsets and principal upper sets $\ua x$. We know the following about Lawson compactness from the literature. 

\begin{theorem}\label{b10}\cite{jia2018meet, xi2017well}
Let $L$ be a dcpo. Then $L$ is Lawson-compact if and only if $L$ is well-filtered, compact and coherent in the Scott topology.
\end{theorem}

\begin{corollary}\label{al-lawson}
Let $L$ be an algebraic domain with bottom and top elements $\bot$ and \rotatebox[origin=c]{180}{$\bot$}, then the following are equivalent:
\begin{enumerate}[(i)]
\item $L$ is not Lawson compact;
\item $L$ is not coherent;
\item $L$ has $\mathcal{K}(C)_\bot$ or $W(E)_\bot^\top$ as a retract, where $C$ is a downward well-ordered chain without a bottom and $E$ is an anti-chain with infinite number of elements.
\end{enumerate}
\end{corollary}
\begin{proof}
Note that an algebraic domain with a least element is well-filtered and compact in the Scott topology. So Lawson compactness is equivalent to coherence by Theorem~\ref{b10}, and the statement is then a straightforward consequence of Theorem~\ref{b8}. 
\end{proof}

\section{Weakly Hausdorff continuous domains}
In this section, we try to generalize the conclusions in Theorem \ref{b8} to the setting of continuous domains. Unfortunately, a satisfactory result like Theorem \ref{b8} \textcolor{black}{has not been achieved yet}. In the \textcolor{black}{continuous setting}, we will have to assume additional properties on the domains, namely a property called \textcolor{black}{weak Hausdorffness} initially introduced by Keimel and Lawson in~\cite{keimel05}, as the title of the section indicates. 

\begin{definition}\cite{keimel05}\label{c4}
A topological space $X$ is \emph{weakly Hausdorff} if for all compact saturated subsets $Q_1,Q_2$ of $X$, and for every open neighborhood $W$ of $Q_1\cap Q_2$, there are two open neighborhoods $U$ of $Q_1$ and $V$ of $Q_2$ such that $W=U\cap V$.\\
A domain $L$ is said to be \emph{weakly Hausdorff} if $\Sigma L$ is a weakly Hausdorff space.
\end{definition}

The following lemma is Lemma 6.6 in~\cite{keimel05}, which follows from a standard compactness argument. 

\begin{lemma}\cite{keimel05}\label{c5}
The following are equivalent for a topological space $X$:
\begin{enumerate}[(i)]
\item $X$ is weakly Hausdorff;
\item for all points $x_1,x_2\in X$, and for every open neighborhood $W$ of $\ua x_1\cap \ua x_2$, there exists an open set $U_1$ containing $x_1$ and an open set $U_2$ containing $x_2$ such that $U_1\cap U_2 \subseteq W$. 
\item {\color{black}for all points $x_1,x_2\in X$, and for every open neighborhood $W$ of $\ua x_1\cap \ua x_2$, there exists an open set $U_1$ containing $x_1$ and an open set $U_2$ containing $x_2$ such that $U_1\cap U_2=W$.}
\end{enumerate}
\end{lemma}

\begin{example}
\
    \begin{enumerate}
        \item Every poset in its Alexandroff topology (the topology consisting of its all upper sets) is weakly Hausdorff (see \cite{goubault2023weakly}). 
        \item Coherent locally compact sober spaces are weakly Hausdorff (see \cite{keimel05}). Since domains are locally compact and sober in the Scott topology, coherent domains are weakly Hausdorff. 
        \item {\color{black}The algebraic domain $W(\mathbb N)^\top_\bot$ in Example 3.2 is weakly Hausdorff: As each element of $W(\mathbb N)^\top_\bot$ is compact, the Scott topology coincides with the Alexandroff topology on $W(\mathbb N)^\top_\bot$, and Lemma 4.2 applies by taking $U=\ua x$ and $V=\ua y$ for every open neighbourhood $W$ containing $\ua x\cap\ua y$; however, it fails to be coherent.}
    \end{enumerate}
\end{example}

    \textcolor{black}{Recall that a topological space $X$ is \emph{Lindel\"of} if every open cover of $X$ has a countable subcover and a topological space $X$ is \emph{hereditarily Lindel\"of} if every subspace of $X$ is Lindel\"of.}

\begin{example}\label{e1}Let $L=\left \{ a_i\mid i\in \mathbb{N}\cup \left \{ \omega \right \}\right \}\cup\left \{ b_i\mid i\in \mathbb{N}\cup \left \{ \omega \right \}\right \}\cup \overline{\mathbb{N}}$. We define an order $\leqslant $ on $L$ as: $x\leqslant y$ in $L$ if and only if,
\begin{itemize}
    \item $x=a_i,\ y=a_j\in\left \{ a_i\mid i\in \mathbb{N}\cup \left \{ \omega \right \}\right \}$ and $i\leqslant j$ in $\mathbb{N}$ or $j=\omega$;
    \item $x=b_i,\ y=b_j\in\left \{ b_i\mid i\in \mathbb{N}\cup \left \{ \omega \right \}\right \}$ and $i\leqslant j$ in $\mathbb{N}$ or $j=\omega$;
    \item $x=a_i\in \left \{ a_i\mid i\in \mathbb{N}\cup \left \{ \omega \right \}\right \}$ and $i\ne \omega$, $y=j\in \overline{\mathbb{N}}$ where $i \leqslant j$;
    \item $x=b_i\in \left \{ b_i\mid i\in \mathbb{N}\cup \left \{ \omega \right \}\right \}$ and $i\ne \omega$, $y=j\in \overline{\mathbb{N}}$ where $i \leqslant j$.
\end{itemize}
The order on $L$ can be depicted as in Figure~5. Then $L$ is a second countable algebraic domain hence {\color{black} $\Sigma L$ is hereditarily Lindel\"of but it is not weakly Hausdorff.}

\begin{figure}[H]
\begin{center}
\begin{tikzpicture}
\coordinate[shape=circle, fill=white,draw=black, inner sep=1pt](a0) at (0, 0); 
\coordinate[shape=circle, fill=white,draw=black, inner sep=1pt](a1) at (0, 1); 
\coordinate[shape=circle, fill=white,draw=black, inner sep=1pt](a2) at (0, 2); 
\coordinate[shape=circle, fill=white,draw=black, inner sep=1pt](an) at (0, 3); 
\coordinate[shape=circle, fill=white,draw=black, inner sep=1pt](aw) at (0, 4);  
			
\coordinate[shape=circle, fill=white,draw=black, inner sep=1pt](b0) at (1, 0); 
\coordinate[shape=circle, fill=white,draw=black, inner sep=1pt](b1) at (1, 1); 
\coordinate[shape=circle, fill=white,draw=black, inner sep=1pt](b2) at (1, 2); 
\coordinate[shape=circle, fill=white,draw=black, inner sep=1pt](bn) at (1, 3); 
\coordinate[shape=circle, fill=white,draw=black, inner sep=1pt](bw) at (1, 4); 
\coordinate[shape=circle, fill=white,draw=black, inner sep=1pt](0) at (3, 4); 
\coordinate[shape=circle, fill=white,draw=black, inner sep=1pt](1) at (4, 4); 
\coordinate[shape=circle, fill=white,draw=black, inner sep=1pt](2) at (5, 4); 
\coordinate[shape=circle, fill=white,draw=black, inner sep=1pt](n) at (7, 4); 		
	\draw (a0) -- (a1);
	\draw (a1) -- (a2);
	\draw[dashed] (a2) -- (an);
	\draw[dashed] (an) -- (aw);
			
	\draw (b0) -- (b1);
	\draw (b1) -- (b2);
	\draw[dashed] (b2) -- (bn);
	\draw[dashed] (bn) -- (bw);
			
	\draw[dashed] (2) -- (n);
	\draw[dashed] (n) -- (9,4);
			
	\draw (a0) -- (0);
	\draw (b0) -- (0);
	\draw (a1) -- (1);
	\draw (b1) -- (1);	
	\draw (a2) -- (2);	        
	\draw (b2) -- (2);	
	\draw (an) -- (n);
	\draw (bn) -- (n);		
			
	\node[left] at (0,0) {$a_0$};
	\node[left] at (0,1) {$a_1$};
	\node[left] at (0,2) {$a_2$};
	\node[left] at (0,3) {$a_n$};
	\node[left] at (0,4) {$a_w$};
			
	\node[left] at (1,0) {$b_0$};
	\node[below left] at (1,1) {$b_1$};
	\node[left] at (1,2) {$b_2$};
	\node[left] at (1,3) {$b_n$};
	\node[left] at (1,4) {$b_w$};
			
	\node[above] at (3,4) {$0$};
	\node[above] at (4,4) {$1$};
	\node[above] at (5,4) {$2$};
	\node[above] at (7,4) {$n$};
\end{tikzpicture}
\\Figure 5: $L$ in Example \ref{e1}
\end{center}
\end{figure}
\end{example}
\begin{proof}
It can be easily seen that all $a_n, b_n, n\in \mathbb N$ are compact elements, so $L$ is \textcolor{black}{an} algebraic domain that is second countable in its Scott topology.\\
To see it is not weakly Hausdorff, we note that $\ua a_{\omega}\cap \ua b_{\omega}=\emptyset$. For every Scott open set $U$ of $a_{\omega}$ and Scott open set $V$ of $b_{\omega}$, there exist  $a_{n_1}\in U$ and $b_{n_2}\in V$. Without loss of generality we assume $n_1\leq n_2$. Then for every $n'\in \overline{\mathbb{N}}$ with $n_2\leqslant n'$, we have $n'\in U\cap V$, which means $L$ is not weakly Hausdorff.
\end{proof}

\begin{theorem}\label{c6}
Let $L$ be a hereditarily Lindel\"of and weakly Hausdorff continuous domain with bottom and top elements $\bot$ and $\top$, then the following are equivalent:
\begin{enumerate}[(i)]
\item $L$ is not coherent;
\item $L$ has $\mathcal{K}(C)_\bot$ or $W(E)_\bot^\top$ as a retract, where $C$ is a downward well-ordered chain without a bottom and $E$ is an anti-chain consisting of {\color{black}a} countably infinite number of elements.
\end{enumerate}
\end{theorem}
\begin{proof}
($\rmnum2\Rightarrow \rmnum1$) Since neither $\mathcal{K}(C)_\bot$ nor $W(E)_\bot^\top$ is coherent, this can be obtained directly from Proposition \ref{b6}.

($\rmnum1\Rightarrow \rmnum2$) Assume that $L$ is not coherent. Since a continuous domain is well-filtered, by using Lemma~\ref{b4} we are able to find $a,b\in L$ such that $\ua a\cap \ua b$ is not compact in $L$. We consider the following two cases: 

$\mathbf{Case\ 1}$: $L$ is not bicomplete. In this case $L$ has some $\mathcal{K}(C)_\bot$ where $C$ is a downward well-ordered chain without a bottom as a retract by using Remark~\ref{bicompelte_bot}.

$\mathbf{Case\ 2}$: $L$ is bicomplete. In this case every element in $\ua a\cap \ua b$ is above some minimal element in $\ua a\cap \ua b$, and  since $\ua a\cap \ua b$ is not compact, those minimal elements in $\ua a\cap \ua b$ form an infinite set. We denote $Z$ as the set consisting of all minimal elements in $\ua a\cap \ua b$, hence $\ua a\cap \ua b = \ua Z$. 

{\color{black}Put $A=\ua a\cap\ua b$. Since $A$ is not compact, there is a family
$\mathcal V$ of Scott open subsets of $L$ which covers $A$ and has no finite
subcover. As $\Sigma L$ is hereditarily Lindelöf, the subspace $A$ is Lindelöf;
hence $\mathcal V$ has a countable subcover of $A$, say $\{V_i\}_{i\in\mathbb N}$.
This countable cover still has no finite subcover, for otherwise $\mathcal V$
would have a finite subcover of $A$.} Without loss of generality, we may assume that each $V_i$ \textcolor{black}{intersects} $Z$; otherwise, we could fix some $V_j$ with $V_j \cap Z \not= \emptyset$, and add $V_j$ to each $V_i$ and then update our sequence of opens. 
From $V_i$, we could construct $U_i, i\in \mathbb N$, such that for every $i\in \mathbb N$,  $U_{i+1}$ is strictly larger than $U_{i}$, $U_{i+1}\cap Z$ is strictly larger than $U_{i}\cap Z$ and $\bigcup_{i=1}^\infty U_i = \bigcup_{i=1}^\infty V_i$. Indeed, we let $U_1 = V_1$ and assume that $U_i$ is defined, then we let ${\color{black}U_{i+1} = U_i \cup \bigcup_{j=1}^{k(i)} V_j}$, where $k(i)$ is the least number such that $V_{k(i)}$ contains an element in $Z\setminus U_i$. This can be done since all $V_i$'s cover $Z$ and any finitely many of them fail to do so. So we also know that $U_i$ fails to cover $Z$ for any $i\in \mathbb N$. 

Now we let $W_1 = U_1$ and fix some element $z_1 \in Z\cap U_1$, and since $U_2 \cap Z$ is strictly larger than $U_1\cap Z$, we fix some $z_2 \in (U_2\setminus U_1)\cap Z$ and let $W_2 = U_2\setminus \da z_1$. Obviously we have $z_2\in W_2$ and $z_2$ differs from $z_1$. We proceed like this and find some element $z_n\in U_n$ different from $z_1, z_2, \cdots z_{n-1}$ and let $W_n = U_n \setminus \da \{z_1, z_2, \cdots z_{n-1}\}$. By induction, we get a sequence of Scott open sets $W_i$ and a sequence of points $z_i\in Z$ with $z_i\in W_i,  i\in \mathbb N$. By the construction, we know that $z_i\in W_j$ if and only $i = j$. As a result, all $W_i$'s are incomparable and they together cover $\ua a\cap \ua b = \ua Z$.

We let $M = \{\top, \bot, a, b \} \cup \{z_i \mid i\in \mathbb N\}$ and give it the induced order from $L$. Then clearly $M$ is a copy of $W(E)_\bot^\top$, where $E$ is countably infinite.

Now we realize that $\ua a\cap \ua b = \ua Z \subseteq \bigcup_{i=1}^\infty W_i$, and then we use the fact that $L$ is weakly Hausdorff to find Scott open sets $U_a$ and $U_b$ such that $a\in U_a, b\in U_b$ and $U_a\cap U_b = \bigcup_{i=1}^\infty W_i$ {\color{black}by Lemma~\ref{c5}.}

Consider the following map $r\colon L\to M$.

\begin{center}
$r(x)= \begin{cases}
\bot, & x\notin U_a\cup U_b  \\
a, &x\in U_a\setminus U_b \\
b, &x\in U_b\setminus U_a \\
z_i, &x\in W_i \setminus \bigcup_{j\not= i} W_j \\
\top, & \mbox{otherwise.}
\end{cases}$
\end{center}

We remark that $r(x) = \top$ if and only if $x$ is in at least two different $W_i$'s. It is quite routine to check that $r$ is a monotone map. We prove that $r$ is Scott continuous. To this end, we let $D$ be a directed subset of $L$ with supremum $\sup D$ and prove $r(\sup D) \leq \sup r(D)$. We consider the following cases:
\begin{itemize}
    \item $r(\sup D) = \bot$. This implies that $\sup D \notin U_a \cup U_b$. {\color{black}Then $r(\sup D)=\bot\leq \sup r(D)$.}
    \item $r(\sup D) = a$. This implies that $\sup D \in U_a\setminus U_b$. Since $U_a$ is Scott open, there exists some $d\in D$ such that $d\in U_a$. {\color{black}This} $d$ is not in $U_b$, since otherwise, $\sup D$ would have been in $U_b$, a contradiction to $\sup D \in U_a\setminus U_b$. So we know \textcolor{black}{$d \in U_a\setminus U_b$, and $r(\sup D) = a = r(d) \leq  \sup r(D)$.} The same reasoning applies to the case $r(\sup D) = b$. 
    \item $r(\sup D) = z_i$ for some $i\in \mathbb N$. Then we know that $\sup D\in W_i \setminus \bigcup_{j\not= i} W_j$, hence some $d\in D$ is in $W_i \setminus \bigcup_{j\not= i} W_j$. So $r(\sup D) = z_i = r(d) \leq \sup r(D)$. 
    \item $r(\sup D) = \top$. In this final case, we would know that $\sup D$ is at least in two different $W_i$'s, say $\sup D\in W_j \cap W_k$ for some $j, k\in \mathbb N$. Then some $d\in D$ is in $W_j \cap W_k$ since  $W_j \cap W_k$ is Scott open. Hence by definition $r(d) = \top$ and $r(\sup D) = \sup r(D) = \top$. 
\end{itemize}
Now we see that $r$ is actually a Scott-continuous retraction, as the canonical embedding $s$ of $M$ into $L$ is monotone and hence Scott-continuous because $M$ is of finite height. Obviously, $r\circ s = \id_M$. 
\end{proof}

The theorem also applies to more special cases listed in the following remark.

\begin{remark}~
\begin{enumerate}
    \item The above theorem also holds if we assume $L$ is weakly Hausdorff and is of countably infinite width, that is, every subset of $L$ consisting of mutually incomparable elements is countable. 
    \item Since second countable spaces are hereditarily Lindel\"of, {\color{black}the above theorem also holds if hereditary Lindel\"ofness is replaced by second countability of $\Sigma L$.}
\end{enumerate}
\end{remark}

\begin{corollary}\label{con-lawson}
Let $L$ be a hereditarily Lindel\"of weakly Hausdorff continuous domain with bottom and top elements $\bot$ and \rotatebox[origin=c]{180}{$\bot$}, then the following are equivalent:
\begin{enumerate}[(i)]
\item $L$ is not Lawson-compact;
\item $L$ is not coherent;
\item $L$ has $\mathcal{K}(C)_\bot$ or $W(E)_\bot^\top$ as a retract, where $C$ is a downward well-ordered chain without a bottom and $E$ is an anti-chain with infinite number of elements.
\end{enumerate}
\end{corollary}
\begin{proof}
This is straightforward from Theorem \ref{b10} and Theorem \ref{c6}.
\end{proof}

{\color{black}
\section{A function space consequence}
In the theory of domains, one usually has to analyze function spaces of domains in order to identify Cartesian closed subcategories of domains, where Cartesian closedness is a necessary property for semantic categories of higher-order programming languages. 

For dcpo's $X,Y$, the function space
$[X\to Y]$ consists of all Scott-continuous maps from $X$ to $Y$,
ordered pointwise: $f\leq g$ if and only $f(x)\leq g(x)$ for all $x\in X$. In this section, a function space $[X\to Y]$ is always endowed with the Scott topology induced by the pointwise order.
Recall that a dcpo $X$ is called an \emph{$L$-dcpo} if every principal ideal
$\da x$ is a complete lattice in its induced order\cite{jung88c}. We shall use the following
results. 

\begin{theorem}\cite[Theorem~4.4.2]{jia2018meet}\label{thm:jia4.4.2}
Let $D$ be a locally compact sober dcpo and $E$ a pointed bicomplete sober dcpo. If $D$ is not coherent and $E$ is not an $L$-dcpo, then the function space $[D\to E]$ is not meet-continuous.
\end{theorem}

\begin{lemma}\cite[Proposition~2.7.10]{jia2018meet}\label{lem:function-space-retract}
Let $P$ be a Scott-continuous retract of a dcpo $L$. Then $\End{P}$ is a
Scott-continuous retract of $\End{L}$.
\end{lemma}

The following standard fact is well-known. 

\begin{lemma}\cite{gierz03}\label{lem:continuous-retract}
Let $P$ be a Scott-continuous retract of a continuous dcpo $L$. Then $P$ is continuous.
\end{lemma}

The following proposition shows that the forbidden structures we find for coherence behave badly with function space construction. 

\begin{proposition}\label{prop:function-spaces-not-continuous}
Let $C$ be a downward well-ordered chain without a bottom element and let $A$
be an infinite anti-chain. Then neither $\End{\mathcal{K}(C)_\bot}$ nor
$\End{W(A)^\top_\bot}$ is continuous.
\end{proposition}
\begin{proof}
It is proved in \cite[Theorem 1.37]{jung88c} that a dcpo with continuous function space is bicomplete. Obviously, $\mathcal{K}(C)_\bot$ is not bicomplete, then the function space $[\mathcal{K}(C)_\bot\to \mathcal{K}(C)_\bot]$ is not continuous.

Let $Q=W(A)^\top_\bot$. Since $Q$ is algebraic, it is continuous and hence locally compact and sober by Proposition~\ref{b2}.
It is also pointed and bicomplete: indeed, $Q$ has finite height, and hence every
filtered subset of $Q$ has a least element. The dcpo $Q$ is not coherent; $Q$ is not an $L$-dcpo,
as in the principal ideal $\da\top=Q$ the two elements $a$ and $b$ do not have a 
least upper bound. Thus, by letting $D=E=Q$, Theorem~\ref{thm:jia4.4.2} implies that $\End{Q}$ is not meet-continuous, hence not continuous.
\end{proof}

\begin{theorem}\label{thm:function-space-consequence}
Let $L$ be either 
\begin{itemize}
    \item a bounded algebraic domain, or
    \item a bounded continuous domain with $\Sigma L$ being hereditarily Lindel\"of and weakly Hausdorff. 
\end{itemize}
If the function space $\End{L}$ is continuous, then $L$ is coherent, or equivalently, Lawson compact.
\end{theorem}
\begin{proof}
We prove this by contradiction and assume that $L$ is not coherent. Then
by Theorem~\ref{b8} and Theorem~\ref{c6}, either $\mathcal{K}(C)_\bot$ for some downward well-ordered chain $C$ without a bottom element, or $W(A)^\top_\bot$ for some infinite anti-chain $A$ appears in $L$ as a Scott-continuous retract of $L$. 
By Lemma~\ref{lem:function-space-retract}, we know that either $\End{\mathcal{K}(C)_\bot}$ or $\End{W(A)^\top_\bot}$ is a Scott-continuous retract of $\End{L}$. As neither $\End{\mathcal{K}(C)_\bot}$ nor $\End{W(A)^\top_\bot}$ is continuous by  Proposition~\ref{prop:function-spaces-not-continuous}, by the contrapositive of Lemma~\ref{lem:continuous-retract} $\End{L}$ is not continuous. This contradicts our assumption and hence $L$ is coherent, or equivalently, Lawson compact by Corollary~\ref{al-lawson} and Corollary~\ref{con-lawson}.
\end{proof}
}

\section{Closing remarks and questions}

We give element-level descriptions for coherence on algebraic domains, and show an algebraic domain $L$ is coherent if and only if $L^\top_\bot$ does not contain copies of $\mathcal K(C)_\bot$ or $W(E)^\top_\bot$ as Scott-continuous retracts (a combination of Theorem~\ref{b8} and Remark~\ref{b9}).  Similar results can be {\color{black}proven for}  domains, but we have to assume the domains are hereditarily Lindel\"of and weakly Hausdorff (Theorem~\ref{c6}). As a byproduct, these results also provide element-level descriptions for Lawson compactness (Corollary~\ref{al-lawson} and Corollary~\ref{con-lawson}). 

{\color{black} In the light of Section~5, we expect that our results could be used in analyzing function spaces of other domains, and we pose the following question:}
\begin{problem}
Does Theorem~\ref{c6} still hold if we remove weak Hausdorffness or  hereditarily Lindel\"ofness in the assumption, or both?    
\end{problem}

\bibliographystyle{elsarticle-num}

\end{document}